\documentclass[11pt]{amsart}
\usepackage{amsmath,amsfonts,amssymb,latexsym}
\usepackage[T1]{fontenc}

\newtheorem{theorem}{\bf \mbox{Theorem}}[section]
\newtheorem{proposition}[theorem]{\bf \mbox{Proposition}}
\newtheorem{lemma}[theorem]{\bf \mbox{Lemma}}
\newtheorem{definition}[theorem]{\bf \mbox{Definition}}
\newtheorem{cor}[theorem]{\bf \mbox{Corollary}}
\newtheorem{remark}{\bf Remark}

\newcommand{\field}[1]{\mathbb{#1}}

\newcommand{\bm}{{\bf m}}
\newcommand{\bn}{{\bf n}}
\newcommand{\bo}{{\bf 0}}
\newcommand{\bp}{{\bf p}}

\newcommand{\bu}{{\bf u}}
\newcommand{\bv}{{\bf v}}
\newcommand{\bx}{{\bf x}}

\newcommand{\Cal}[1]{\mathcal{#1}}

\newcommand{\CC}{\mathbb{C}^2}

\newcommand{\cO}{\Cal{O}}
\newcommand{\dist}{{\rm dist}}

\newcommand{\Q}{\field{Q}}
\newcommand{\R}{\field{R}}

\newcommand{\vp}{{\varphi}}

\def\C{\mathbb{C}}
\def\K{\mathbb{K}}

\begin{document}
\title{On the bi-Lipschitz contact equivalence of plane complex function-germs}
\author{- Lev Birbrair - Alexandre Fernandes - Vincent Grandjean -}
\address{ Departamento de Matem\'atica, UFC,
Av. Humberto Monte s/n, Campus do Pici Bloco 914,
CEP 60.455-760, Fortaleza-CE, Brasil}
\email{birb@ufc.br}
\email{alexandre.fernandes@ufc.br}
\email{vgrandje@fields.utoronto.ca}
\thanks{}
\subjclass[2010]{Primary ,Secondary }

\date{\today}
\maketitle
%
%
%

\begin{abstract} In this note, we consider the problem of bi-Lipschitz contact equivalence of
complex analytic function-germs of two variables. Basically, it is inquiring about the infinitesimal sizes
of such function-germs up to bi-Lipschitz changes of coordinates. We show that this problem is equivalent
to right topological classification of such function-germs.
\end{abstract}
\section{Contact equivalence}

\medskip
Two $\K$-analytic function-germs $f,g\colon (\K^{n},\bo)\rightarrow(\K,0)$, at the origin $\bo$ of $\K^n$, are
($\K$-analytically) \em contact equivalent \em if the ideals
(in ${\Cal O}_{\K^n,\bo}$) generated by $f$ and, respectively, generated by $g$ are $\K$-analytically
isomorphic. As is well known, this classical ($\K$-analytic) contact equivalence admits moduli.
For a complete description and answer to Zariski \em probl\`eme des modules pour les branches planes \em in the uni-branch case,
see \cite{HH},  (see also \cite{HHH} for an answer towards the general case).
Over the years several generalizations of the notion of ($\K$-analytic) contact equivalence appeared,
and for some rough ones moduli do not appear.

\medskip
%
More precisely, we will say that two function-germs $f,g\colon (\K^{n},\bo)\rightarrow(\K,0)$ 
at the origin $\bo$ of $\K^n$ are \emph{bi-Lipschitz contact equivalent} if there exists 
$H\colon(\K^{n},\bo)\rightarrow(\K^{n},\bo)$ a bi-Lipschitz homeomorphism  and there exist 
positive constants $A$ and $B$, and $\sigma \in\{-1,+1\}$ such that
\begin{center}
\vspace{4pt}
$A|f(\bp)| \leq |g\circ H(\bp)| \leq  B|f(\bp)|$  when $\K = \C$,
\vspace{6pt}
\\
$A f(\bp)  \leq  \sigma \cdot (g\circ H(\bp))  \leq  B f(\bp)$  when $\K = \R$,
\vspace{4pt}
\end{center}
for any point $\bp\in \K^n$ close to $\bo$.

When the bi-Lipschitz homeomorphism $H$ is also subanalytic, we will say that the functions $f$ and $g$ are
\emph{subanalytically bi-Lipschitz contact equivalent}.

\medskip
A consequence of the main result of \cite{BCFR} on bi-Lipschitz contact equivalence of
Lipschitz function-germs is the following finiteness

\smallskip\noindent
{\bf Theorem} (\cite{BCFR}). \em For any given pair $n$ and $k$ of positive integers,
the subspace of polynomial function-germs $(\K^{n},\bo)\rightarrow(\K,0)$ of degree smaller than or
equal to $k$ has finitely many bi-Lipschitz contact equivalence classes. \em

\medskip
Later on, Ruas and Valette (see \cite{RV}) obtained for real mappings a result more general than that of \cite{BCFR}, and which
again ensures the finiteness of the bi-Lipschitz contact equivalent classes for polynomial functions germ $(\K^{n},\bo)
\rightarrow(\K,0)$ with given bounded degree.
However, we observe that in the aforementioned papers \cite{BCFR,RV}, the proofs of the finiteness theorems
for bi-Lipschitz contact equivalence do not say anything about the corresponding recognition problem.

\medskip
The preprint \cite{BFGG} completely solves the recognition problem of subanalytic contact bi-Lipschitz equivalence
for continuous subanalytic function-germs $(\R^{2},\bo)\rightarrow(\R,0)$ by providing an explicit
combinatorial object which completely characterizes the corresponding orbit.

\medskip
In the present note, we solve the recognition problem for the subanalytic bi-Lipschitz contact equivalence of complex
analytic function-germs $(\C^{2},\bo)\rightarrow(\C,0)$.

\smallskip
Our main result, Theorem \ref{thm:main}, states that the subanalytic bi-Lipschitz contact equivalence 
class of a plane complex analytic function-germ $f\colon(\C^{2},\bo)\rightarrow(\C,0)$ determines and 
is determined by purely numerical data, namely: \em the Puiseux pairs of each branch of its zero locus, 
the multiplicities of its irreducible factors and the intersection numbers of pairs of branches of its zero locus. \em
It is a consequence of Theorem \ref{thm:order} which explicits the order of an irreducible function-germ $g$ 
along real analytic half-branches at $\bo$ as an affine function of the contact of the half-branch and the zero 
locus of $g$ the function-germ.  

\smallskip
Last, combining the main result of \cite{Pa} and our main result, we eventually get that two complex analytic function
germs $f,g\colon(\C^{2},\bo)\rightarrow(\C,0)$ are subanalytically bi-Lipschitz contact equivalent if, and only if, they are
\em right topologically equivalent, \em i.e. there exists a homeomorphism $\Phi\colon(\CC,\bo)\rightarrow(\CC,\bo)$
such that $f=g\circ \Phi$.
\section{Preliminaries}
We present below some well known material about complex analytic plane curve-germs.
It will be used in the description and the proof of our main result.
\subsection{Embedded topology of complex plane curves}\label{subsection:puiseux}
$ $

\smallskip
Let $f\colon(\C^{2},\bo)\rightarrow(\C,0)$ be the germ at $\bo$ of an irreducible analytic function.
It admits a Puiseux parameterization of the following kind:
\begin{equation}\label{eq:Puiseux-param}
x\to (x^m, \Psi (x)) \; \mbox{\rm with } \Psi (x) = x^{\beta_1}\vp_1(x^{e_1}) +
\ldots +  x^{\beta_s}\vp_s (x^{e_s}),
\end{equation}
where each function $\vp_i$ is a holomorphic unit at $x=0$, the integer number $m$ is the multiplicity of the function $f$ 
at the origin and $(\beta_1,e_1),\ldots,(\beta_s,e_s)$ are the Puiseux pairs of $f$. Then we can write down,
\begin{equation}\label{eq:fact-irred}
f (x^m,y) = U(x,y)  \Pi_{i=1}^m (y - \Psi(\omega^i x)),
\end{equation}
where $\omega$ is a primitive $m$-th root of unity, the function $U$ is a holomorphic unit at the origin,
and $\Psi$ is a function like in Equation (\ref{eq:Puiseux-param}).

\medskip
The following relations determines the Puiseux pairs of $f$. Let us write $\Psi (x) = \sum_{j>m} a_jx^j$ 
and $e_0 := m$ and $\beta_{s+1} := + \infty$. We recall that
\begin{center}
\vspace{4pt}
$\beta_{i+1} = \min\{j: a_j \neq 0 \mbox{ and } e_i \not | j\}$
and $e_{i+1} := \gcd (e_i,\beta_{i+1})$
\vspace{4pt}
\end{center}
for $i=0,\ldots,s-1$. We deduce there exists positive integers $m_1,\ldots, m_s,$ such that
for each $k=1,\ldots,s$, we find 
\begin{equation}\label{eq:multiplicity}
m = e_1 m_1 = e_2 m_2 m_1 = \ldots = e_k (m_k \cdots m_1) 
\end{equation}
We recall that the irreducibility of the function $f$ implies that $e_s=1$.
\begin{remark}
Let $f\colon(\CC,\bo)\rightarrow(\C,0)$ be an irreducible analytic function-germ and let $X$ be its zero locus.
The ideal $I_X$ of $\C\{x,y\}$ consisting of all
the functions vanishing on $X$ is generated by $f$. If $g = \lambda f$ is any other generator of $I_X$, then the functions
$f$ and $g$ have the same Puiseux pairs. Thus we will speak of the \emph{Puiseux pairs of the branch} $X$.
\end{remark}
Let $f_1,f_2\colon(\CC,\bo)\rightarrow(\C,0)$ be irreducible analytic function-germs, and let $X_1$ and $X_2$ be
the respective zero sets of $f_1$ and $f_2$.

The \emph{intersection number at $\bo$} of the branches $X_1$ and $X_2$ is defined as:
$$(X_1,X_2)_{\bo}=\dim_{\C}\frac{\C\{x,y\}}{(f_1,f_2)}$$
where $(f_1,f_2)$ denotes the ideal generated by $f_1$ and $f_2$.

\medskip\noindent
{\bf Notation:} Let $\Phi: (\CC,\bo) \to (\CC,\bo)$ be a homeomorphism and let $X$ be a subset germ of
$(\CC,\bo)$. We will write
$$\Phi: (\CC,X,\bo) \to (\CC,Y,\bo)$$
to mean that the subset germ $Y$ is the germ of the image $\Phi (X)$ of $X$.

\medskip
The following classical result completely described the classification of embedded complex plane curve germs:
\begin{theorem}[\cite{Bu,Za}]\label{theor-classic}
Let $f,g\colon(\CC,\bo)\rightarrow(\C,0)$ be reduced analytic function-germs and let $X$ and $Y$ be the respective zero
sets of $f$ and $g$. Let $X=\bigcup_{i=1}^rX_i$ and $Y=\bigcup_{i=1}^sY_i$ be the irreducible components of $X$ and $Y$
respectively. There exists a  homeomorphism $\Phi\colon(\CC,X,\bo)\rightarrow(\CC,Y,\bo)$ if and only if,
up to a re-indexation of the branches of $Y$, the components $X_i$ and $Y_i$ have the same Puiseux pairs, and each pair of
branches $X_i$ and $X_j$ have the same intersection numbers as the pair $Y_i$ and $Y_j$.
\end{theorem}
We end-up this subsection in recalling a recent result of Parusi\'nski \cite{Pa}. It is as much a generalization of
Theorem \ref{theor-classic} to the non reduced case, as it is an improvement in the sense that it provides a
more rigid statement.
\begin{theorem}\label{theor-parusinski} Let $f,g\colon(\C^{2},\bo)\rightarrow(\C,0)$ be complex analytic
function-germs (thus not necessarily reduced). There exists a germ of homeomorphism $\Phi: (\CC,\bo)\to(\CC,\bo)$ such
that $g\circ\Phi = f$ (the function-germs $f$ and $g$ are then said \em topologically right-equivalent\em) if, and only if,
there exists a one-to-one correspondence between the irreducible factors of $f$ and $g$ which preserves
the multiplicities of these factors, their Puiseux pairs and the intersection numbers of any pairs of distinct
irreducible components of the respective zero loci of $f$ and $g$.
\end{theorem}
\subsection{Lipschitz geometry of complex plane curve singularities}
$ $

\smallskip
The Lipschitz geometry of complex plane curve singularities we are interested in
is the Lipschitz geometry which comes from being embedded in the plane.
It is described in a collection of three articles over 40 years, initiated with the seminal paper \cite{PT}, followed
then  by \cite{Fe} and concluding for now with the recent preprint \cite{NP}.
Those papers state that the Lipschitz geometry of complex plane curve singularities
determines and is determined by the embedded topology of such singularities. The version of this result which
we are going to use is the following one:
\begin{theorem}\label{theor-curves}
Let $X$ and $Y$ be germs of complex analytic plane curves at $\bo\in\CC$. Then, there exists a homeomorphism
$\Phi\colon(\CC,X,\bo)\rightarrow(\CC,Y,\bo)$
if, and only if, there exists a (subanalytic) bi-Lipschitz homeomorphism
$H\colon(\CC,X,\bo)\rightarrow(\CC,Y,\bo).$
\end{theorem}
The version stated above is almost Theorem 1.1 of \cite{NP}. The exact statement of Theorem 1.1 of \cite{NP} does not
require the subanalyticity of the homeomorphism $H$. However, we observe that the proof presented there
actually guarantees the subanalyticity of the mapping $H$.
%
%
%
%
%
%
%
%
%
%
%
%
%
%
%
%
%
%
%
%
%
%
%
%
%
%
%
%
%
%
%
%
\section{On the irreducible functions case}\label{section:irreducible}
This section is devoted to the relation between the order of a given irreducible plane complex 
function-germ $f$ along any real analytic half-branch germ at the origin $\bo$ of $\C^2$, and  
the contact (at the origin) between the half-branch and the zero locus $X$ of $f$. 
(Both notions of order and contact will be recalled below.) Theorem \ref{thm:order} is the main result of
the section and the key new ingredient to complete the subanalytic bi-Lipschitz contact classification. It 
states that the contact and the order satisfies an affine relation whose coefficients can be 
explicitly computed by means of the Puiseux data of $X$ presented in sub-Section \ref{subsection:puiseux}. 
 
\medskip
We suppose given some local coordinates $(w,y)$ centered at the origin of $\C^2$. 

\smallskip
Let $\Gamma$ be a real-analytic half branch germ at the origin of $\C^2$, that is the image 
of (the restriction of) a real analytic map-germ $\gamma: (\R_+,0) \to (\C^2,\bo)$ defined as 
$s\to \gamma(s) = (w(s),y(s))$. 
When $\Gamma$ is not  contained in the $y$-axis, 
we can assume that $\gamma (s) = (s^e \bu (s), s^{e'} \bv(s))$ 
for positive integers $e,e'$ with $\bu(z),\bv(z) \in \cO_1:=\C\{z\}$ and $\bu (0),\bv(0) \neq 0$.

\smallskip
When $\Gamma$ is not contained in the $y$-axis, we want to find a holomorphic change of coordinates $w \to x(w)$ so that 
\begin{equation}\label{eq:change}
x(z^e \bu (z)) = z^e \iff \bu(z)\cdot\bx (z^e \bu(z)) = 1 
\end{equation}
writing $x$ as $x(w) := w \cdot \bx (w)$ for a local holomorphic unit $\bx$.
Thus Equation (\ref{eq:change}) admits a holomorphic solution. The mapping $\Theta:(w,y) \to 
(x(w),y) = (x,y)$ is bi-holomorphic in a neighbourhood of the origin. 
In the new coordinates $(x,y)$, the mapping $\gamma$ now writes as $s \to (s^e, s^{e'} \bv (s))$.

\smallskip\noindent
{\bf Vocabulary}. A map-germ $\phi:(\R_+,0) \to (\C^2,\bo)$ is \em ramified analytic \em if there exists 
a function germ $\tilde{\phi} \in \cO_1$ and (co-prime) positive integers $p,q$ such that $\phi (t) = \tilde{\phi}(t^{p/q})$. 
We will further say that $\phi$ is \em a ramified analytic unit \em if $\tilde{\phi}$ is a holomorphic unit.

\smallskip
When $\Gamma$ is not contained in the $y$-axis, we re-parameterize $\gamma$ with $s (t) := t^{e/m}$
for $t \in \R_+$, so that $\gamma(t): = \gamma (s(t)) = (t^m,y(t))$ where $y$ is ramified analytic with 
$y(0)=0$ and $m$ is the multiplicity of the function $f$ at the origin.

\smallskip
If $\Gamma$ is contained in the $y$-axis then we take $s=t$ and $\Theta$ is just the identity mapping.

\medskip
We recall that the Puiseux pairs introduced in sub-Section \ref{subsection:puiseux} are bi-holomorphic 
invariant. 
We denote again $f = f(x,y)$ for $f\circ \Theta^{-1}$ and use the Puiseux decomposition for 
$f(x^m,y)$ given in Equation (\ref{eq:fact-irred}) to define for each 
$k = 0,\ldots,s$, the function germ $\Psi_k \in \cO_1$ as 
\begin{center}
\vspace{4pt}
\begin{tabular}{rcl}
\vspace{4pt}
$\Psi_0(x)$ & := & $0$, 
\\
\vspace{4pt}
$\Psi_k(x)$ & := & $x^{\beta_1}\vp_1(x^{e_1}) + \ldots + x^{\beta_k}\vp_k(x^{e_k})$ when $k\geq 1$.
\end{tabular}
\end{center}
Note that $\Psi_k (x) = \theta_k (x^{e_k})$ for some function germ $\theta_k \in \cO_1$.

\smallskip
For each $l=1,\ldots,m$, we can write 
\begin{center}
$y(t) = \Psi (\omega^l t) + t^{\lambda_l} u_l(t)$
\end{center}
where $\lambda_l \in \Q_{>0}\cup\{+\infty\}$ for $u_l$ is a ramified analytic unit, and with the 
convention that we write the null function $0$ as $0 = t^{+\infty} u_l (t)$. Thus the half-branch 
$\Gamma$ is contained in $X$ if and only if there exists $l$  such that $\lambda_l = +\infty$.

\medskip\noindent
{\bf Notation.} 
\em Let $\lambda : = \max_{l=1,\ldots,m}\lambda_l$.
\em

\smallskip
Let $l \in \{1,\ldots,m\}$ so that $\lambda = \lambda_l$. When $\Gamma$ is not contained in $X$ 
(equivalently $\lambda < +\infty$) and convening further that $\beta_0 =0$ and $\beta_{s+1} = + \infty$, 
there exists a unique integer $k\in \{0,\ldots,s\}$ such that 
\begin{center}
$\beta_k\leq \lambda <\beta_{k+1}$, 
\end{center}
and consequently we can write 
\begin{center}
$y(t) = \Psi_k (\omega^l t) + t^\lambda u(t)$
\end{center}
for $u$ a ramified analytic unit.
(Note that $\Psi = \Psi_k + R_k$ where $R_k(x) = (\Psi - \Psi_k)(x) = O (x^{\beta_{k+1}})$.) 

\smallskip
Evaluating the function $f$ along the parameterized arc $t \to \gamma (t)$ using Equation (\ref{eq:fact-irred}) gives
\begin{center}
\vspace{4pt}
$f(\gamma(t)) = f(t^m,y(t)) = f(t^m, \Psi_k(t) + t^\lambda u(t)) = 
U(t)\Pi_{i=1}^m[\Psi_k(\omega^l t) + t^\lambda u(t) - \Psi(\omega^i t)]$
\vspace{4pt} 
\end{center}
where $t\to U(t)$ is a ramified analytic unit.
Since the function $t \to f(\gamma (t))$ is a ramified analytic function, there exist a ramified analytic unit  
$V $ and a number $\nu\in \Q_{>0}\cup\{+\infty\}$ such that 
\begin{equation}\label{eq:order}
f(\gamma(t)) = t^\nu V(t). 
\end{equation}
The number $\nu$ of Equation (\ref{eq:order}) is called the \em order of the function $f$ along the  
parameterized curve $t\to \gamma (t)$. \em
\begin{lemma}\label{lem:order-arc}
1) Assume $\Gamma$ is contained in the $y$-axis. The order of the function $f$ along the parameterized curve 
$t\to\gamma(t) = (0,t^{e'}\bv(t))$ is $\nu = m \cdot e'$,

\smallskip\noindent
2) Assume $\Gamma$ is not contained in the $y$-axis. 
The order of $\nu$ the function $f$ along the parameterized curve $t\to\gamma(t)$ is given by 
$$
\nu = e_k \lambda + (e_0-e_1)\beta_1 + \ldots + (e_{k-1} - e_k) \beta_k \in \Q_{>0}\cup\{+\infty\}
$$
\end{lemma}
\begin{proof}
If $\Gamma$ is contained in the $y$-axis, then the order of $f$ along $t\to (0,t^{e'}\bv (t))$ is $m\cdot e'$.

\medskip
We can assume that $\Gamma$ is parameterized as $\R_+ \ni t \to \gamma(t) = (t^m, \psi_k(t) + t^\lambda u(t))$.

\smallskip
For $i\in \{1,\ldots,m\}$ such that $l-i$ is not a multiple of $m_1$, the order of $\Psi_k(\omega^l t) + t^\lambda u(t) - 
\Psi(\omega^i t)$ is $\beta_1$. There are $m-1 -(e_1-1) = e_0 - e_1$ such indices $i$.

\smallskip
For any $0<j <k$, when $i \in \{1,\ldots,m-1\}$ is such that $l-i$ is a multiple of $m_1 \ldots m_j$ but not a multiple 
of $m_1 \ldots m_{j+1}$, the order of $\Psi_k(\omega^l t) + t^\lambda u(t) - \Psi(\omega^i t)$ is $\beta_j$. 
There are $e_j - e_{j+1}$ such indices. 

\smallskip
When $i \in \{1,\ldots,m\}$ is such that $l-i$ a multiple of $m_1 \ldots m_k$, the order of  
$\Psi_k(\omega^l t) + t^\lambda u(t) - \Psi(\omega^i t)$ is $\lambda$. There are $e_k$ such indices.

\smallskip
We just add-up all these orders to get the desired number $\nu$, once we have checked that this sum does not depend 
on the index $l$ such that $\lambda = \lambda_l$. Let $r\in\{1,\ldots,m\}$ be an index such that $\lambda_r = \lambda$.
Thus $y(t) = \Psi_k(w^r t) + t^\lambda u_r (t)$. 
If $l-r$ is not a multiple of $m_1 \cdots m_k$, then we check again that 
$0 = y(t) - y(t) = t^\lambda (u_l (t) - u_r (t)) + t^{\beta_j} W$ for a ramified analytic unit $W$ and 
$\beta_j \leq \beta_{k-1} < \lambda$, which is impossible.    
Necessarily $l-r$ is a multiple of $m_1\cdots m_k$ and thus $\Psi_k (w^r t) = \Psi_k (\omega^l t)$, so that 
$\nu$ is well defined.
\end{proof}

\medskip
Now we can introduce a sort of normalized parameterization of real analytic half-branch 
germs in order to do bi-Lipschitz geometry. More precisely, 
\begin{definition}
An \em analytic arc \em (at the origin of $\C^2$) is the germ at $0\in \R_+$ of a mapping $\alpha: [0,\epsilon[\rightarrow \CC$ 
defined as $t \to (x(t),y(t))$ such that:

0) the mapping $\alpha$ is not constant and $\alpha (0) = \bo$,

1) there exist a positive integer $e$ such that $t\to \alpha (t^e)$ is (the restriction of) a real analytic
mapping,

2) the arc is parameterized by the distance to the origin in the following sense: there exists positive constants
$a<b$ such that for $0\leq t \ll 1$ the following inequalities hold,
\begin{center}
$a t \leq |\alpha(t)| \leq bt.$
\end{center}
\end{definition}

\smallskip\noindent
We will denote any analytic arc by its defining mapping $\alpha$. 
Note that the semi-analyticity of the image of an analytic arc $\alpha$ implies a much better
asymptotic than that proposed in the definition, namely we know that
$|\alpha (t)| = \alpha_1 t + t \delta (t)$,
with $\alpha_1 >0$ and where $\delta$ is ramified analytic such that $\delta (0) = 0$.

\medskip
Let $\alpha$ be a real analytic arc. The function $t\to f\circ \alpha (t)$ is ramified analytic, thus 
as already seen in Equation (\ref{eq:order}) can be written as $f\circ \alpha (t)= t^{\nu_f (\alpha)} V(t)$ 
for a ramified analytic function and $\nu_f (\alpha) \in \Q_{>0}\cup\{+\infty\}$.
The \em order of the function $f$ along the real analytic arc $\alpha$ \em is the well defined rational number
$\nu_f (\alpha)$.

\smallskip
Let $C$ be a real-analytic half-branch germ at the origin of $\C^2$. 
Let $\alpha$ and $\beta$ be two real analytic arcs parameterizing $C$. 
We check with an easy computation that $\nu_f(\alpha) = \nu_f (\beta)$. 
Thus we introduce the following
\begin{definition}
The \em order of the function $f$ along the real analytic half-branch $C$ \em is the well defined number
$\nu_f(C) := \nu_f (\delta)$ for any arc $\delta$ parameterizing $C$.
\end{definition}

\smallskip
Let us denote $X(r)=\{\bp\in X \ : |\bp|=r\}$ for $r$ a positive real number.

\smallskip
Let $\alpha$ be any analytic arc. The \em contact \em{(}\em at the origin\em{) }\em between the analytic arc
$\alpha$ and the complex curve-germ $X$ \em is the rational number defined as
$$
c(\alpha,X)= 
\lim_{t\to 0+}\frac{\log(\dist(\alpha(t),X(|\alpha (t)|)))}{\log(t)}.
$$
Let $C$ be the image of the analytic arc $\alpha$ above. Given any other analytic arc $\beta$ 
parameterizing $C$, it is a matter of elementary computations to check that 
$c(\alpha,X) = c(\beta,X)$. Thus we present the following
\begin{definition}
The \em contact between the real-analytic half-branch $C$ and the curve $X$ \em is 
$c(C,X) := c(\delta,X)$ for any analytic arc $\delta$ parameterizing $C$.
\end{definition}
Let $\Gamma$ be a real analytic half-branch at the origin of $\C^2$. 
Let $\gamma$ be a parameterization of $\Gamma$ of the form $\R_+ \ni t \to (0,y(t))$ when $\Gamma$ 
is contained in the $y$-axis, where $y$ is a ramified analytic function-germ.
When $\Gamma$ is not contained in the $y$-axis, possibly after a holomorphic change of coordinates at 
the origin of $\C^2$, 
we consider a parameterization of $\Gamma$ of the form $\R_+\ni t\to (t^m,y(t))$ for $y$ ramified analytic.  

\smallskip
When the half-branch $\Gamma$ is not contained in $X$ (and regardless of its position relatively to
the $y$-axis), as already seen above, 
we can write $y(t)$ as $y(t) = \Psi_k (\omega^l t) + t^\lambda u(t)$ where $\beta_k \leq \lambda < \beta_{k+1}$ 
for some integer $k\in \{0,\ldots,s\}$, with $u$ a ramified analytic unit and $l \in \{1,\ldots,m\}$. 
Let $\mu$ be the order of $|\gamma(t)|$ at $t=0$, that is the positive rational number $\mu$ such
that $|\gamma(t)| = M t^\mu + o(t^\mu)$ for a positive constant $M$. 
Thus we find 
\begin{lemma}\label{lem:contact}
The contact between $\Gamma$ and $X$ is $c(\Gamma,X) = \frac{\lambda}{\mu}$.
\end{lemma}
\begin{proof}
The tangent cone at the origin of the curve $X$ is just the $x$-axis. 
Writing $\gamma(t) = (x(t),y(t))$, the half-branch is tangent to the $x$-axis if and only if 
$\lim_{t\to 0} x(t)^{-1}y(t) = 0$. When $\Gamma$ is transverse to the $x$-axis, we have $k = 0$ 
in the writing of $y(t)$ above, so that $\mu = \lambda$ and thus $c(\Gamma,X) = 1$. 

\medskip
Suppose that the half-branch $\Gamma$ is tangent to the $x$-axis, we deduce $\mu = m$
since the tangency hypothesis implies that $y(t) = o(t^m)$. Thus the mapping 
$t\to \gamma(t^\frac{1}{m}) = (t,y(t^\frac{1}{m})$ is an analytic arc parameterizing $\Gamma$. 
In particular we must have $\lambda > m$.

\medskip\noindent
{\bf Notation.} Up to the end of this proof will use the notation $Const$ to mean a 
positive constant we do not want to precise further.

\smallskip
Let $\rho:(\R_+,0) \to (\R_+,0)$ be the function defined as 
$\rho (t) := \dist (\gamma(t^\frac{1}{m}),X)$. First, since $\gamma$ is tangent to $X$ and the 
function $\rho$ is continuous and subanalytic, there exists a positive rational number $c$ such that 
\begin{equation}\label{eq:bound1}
\rho(t) = Const\cdot t^c + o(t^c). 
\end{equation}
Second, we obviously have for $t$ positive and small enough $\rho (t)\leq  t^\frac{\lambda}{m}|u(t)|$
so that we deduce from Equation (\ref{eq:bound1}) that $c \geq \frac{\lambda}{m}$.

\smallskip
Let $r(t) := |\gamma (t^\frac{1}{m})|$, so that we find $r(t) = t + o(t)$. Let $t\to\phi (t)$ 
be any analytic arc on $X$ such that $\rho (t) = |\phi(t) - \gamma (t^\frac{1}{m})|$. From 
Equation (\ref{eq:bound1}) we get
\begin{equation}\label{eq:bound2}
||\phi(t)| - r(t) | \leq Const \cdot t^c. 
\end{equation}
Writing $\phi = (x_\phi,y_\phi)$, we see from Equation (\ref{eq:bound2}) that $x_\phi (t) = t + O(t^c)$. 
Let $\xi:(\R_+,0) \to (\C,0)$ be the ramified analytic function of the form $t\to\xi(t) : = t^\frac{1}{m}[1 + O(t^{c-1})]$ 
and such that $\xi(t)$ is a $m$-th root of $x_\phi(t)$. 
Thus $y_\phi (t) = \Psi (\omega^i \xi (t))$ for some $i\in\{1,\ldots,m\}$ and we observe that 
$y_\phi(t) = \Psi (\omega^i t^\frac{1}{m}) + o(t^\frac{\lambda}{m})$. 
Since $y(t) = \Psi_k (\omega^l t^\frac{1}{m}) + t^\frac{\lambda}{m}u(t^\frac{1}{m})$, with $u$ a ramified 
analytic function, and $|y_\phi (t) - y(t^\frac{1}{m})| \leq Const \cdot t^c$, we deduce that 
$\Psi_k (\omega^i T) = \Psi_k (\omega^l T)$. But this implies that $c\leq  \frac{\lambda}{m}$, and thus
$c= \frac{\lambda}{m}$.

\smallskip
From Equation (\ref{eq:bound2}) we deduce that 
\begin{equation}\label{eq:bound3}
\rho (t) \leq \dist (\gamma(t^\frac{1}{m}),X(r(t)) \leq Const \cdot t^c.
\end{equation}
Combining Equation (\ref{eq:bound1}) and Equation (\ref{eq:bound3}) we get the result.\
\end{proof}

The next result will be key for Theorem \ref{thm:main}, the main result of this note, 
is indeed the new ingredient to the range of questions we are dealing with here. 
We recall that the Puiseux data notation convenes that $e_{-1} = \beta_0 =0$, $e_0 = m$ 
and $\beta_{s+1} = +\infty$.
\begin{theorem}\label{thm:order}
Let $\Gamma$ be a real analytic half-branch at the origin of $\C^2$ as above. 
The order of the function $f$ along $\Gamma$ is given by 
\begin{eqnarray}\label{1}
\nu_f (\Gamma)  & = & e_k \cdot c(\Gamma,X) + (e_0 -e_1)\frac{\beta_1}{m} + \ldots + (e_{k-1} -e_k) \frac{\beta_k}{m} 
\\
\label{2}
& = & e_k \left(c(\Gamma,X) - \frac{\beta_k}{m}\right) + \sum_{i=\min (k-1,0)}^{k-1} e_i\left(\frac{\beta_{i+1}}{m} -\frac{\beta_i}{m}\right),
\end{eqnarray}
where the integer number $k \in \{0,\ldots,s\}$ in Equations (\ref{1}) and (\ref{2}) is uniquely determined when $c(\Gamma,X) < + \infty$ 
by the following condition:
\begin{center}
$\beta_k \leq m\cdot c < \beta_{k+1}$.
\end{center}
\end{theorem}
\begin{proof}
It is just a rewriting of Lemma \ref{lem:order-arc} in term of the size $t$ of any arc parameterizing $\Gamma$ 
and uses Lemma \ref{lem:contact}.
\end{proof}

A direct consequence of the above result is the following result about
bi-Lipschitz contact equivalence.

\begin{proposition}\label{main-proposition}
Let $(\CC,X,\bo)$ and $(\CC,Y,\bo)$ be two germs of irreducible complex plane
curves defined by reduced function-germs $f$ and $g$ respectively. If there exists a subanalytic bi-Lipschitz
homeomorphism $H\colon (\CC,X,\bo)\rightarrow (\CC,Y,\bo)$ then there exist positive constants $0<A<B<+\infty$
such that in a neighbourhood of the origin we find
$$
A |f| \leq |g \circ H| \leq B |f|.
$$
\end{proposition}
\begin{proof}
If it is not true, it happens along a real-analytic half-branch $C$. Necessarily such a half-branch $C$ must be
tangent to the curve $X$. Taking a parameterization of $C$ by an arc $\alpha$, we can for instance assume that 
$(f\circ\alpha(t))^{-1} (g\circ H\circ \alpha (t))$ goes to $0$ as $t$ goes to $0$.
Let $\nu$ be the order of $f(\alpha(t))$ and $\nu'$ the order of $g(H(\alpha(t)))$. 
Theorem \ref{thm:order} provides  
\begin{eqnarray*}
\nu & = &  (e_0 -e_1)\frac{\beta_1}{m} + \ldots + (e_{k-1} -e_k) \frac{\beta_k}{m} + e_k \cdot c(C,X) \\
\nu' & = & (e_0 -e_1)\frac{\beta_1}{m} + \ldots + (e_{k'-1} -e_{k'}) \frac{\beta_{k'}}{m} + e_{k'} \cdot c(H^{-1}(C),Y).
\end{eqnarray*}
From the proofs of Lemma \ref{lem:order-arc} and Lemma \ref{lem:contact} we know that   
\begin{center}
$\beta_{k'} \leq m \cdot c(H^{-1}(C),Y) < \beta_{k'+1}$ and $\beta_k \leq  m\cdot c(C,X) < \beta_{k+1}$.
\end{center}
Since the contact is a bi-Lipschitz invariant we get $c(C,X) = c(H^{-1}(C),Y)$. Besides $\nu' > \nu$, thus 
we deduce $k' > k$. This latter inequality implies  
\begin{center}
$m \cdot c(H^{-1}(C),Y) \geq \beta_{k'} \geq \beta_{k+1} > m\cdot c(C,X)$,
\end{center}
which is impossible. 
\end{proof}
\section{Main Result}
Let $f\colon(\CC,\bo)\rightarrow(\C,0)$ be a germ of analytic function. Let $f=f_1^{\bm_1}\cdots f_r^{\bm_r}$ be the irreducible
decomposition of the function, where $f_1,\dots,f_r$ are irreducible function-germs and $\bm_1,\dots,\bm_r$, the corresponding
respective multiplicities, are positive integer numbers.

\medskip
Let $X_i$ be the zero locus of $f_i$, let $m_i$ be the multiplicity of $f_i$ at $\bo$ and let $(\beta_j^{(i)},e_j^{(i)})_{j=1}^{s_j}$
be its Puiseux pairs.
Let $\Gamma$ be a real analytic half-branch at the origin. Let $c_i := c(\gamma,X_i)$ be the contact of $\Gamma$ with $X_i$ 
and let $\nu_i = \nu_{f_i}(\Gamma)$ be the order of $f_i$ along $\Gamma$.

Since we have defined in Section \ref{section:irreducible} the order of an irreducible function-germ along $\Gamma$,
the order of $f$ along $\Gamma$ is defined as the sum of the order of each of its irreducible 
component weighted by the corresponding multiplicity (as a factor of the irreducible decomposition of $f$).   
From Theorem \ref{thm:main} we deduce straightforwardly the next
\begin{lemma}\label{lem:order-general}
The order $\nu$ of the function $f$ along $\Gamma$ is
\begin{eqnarray*}
\nu & := & \bm_1 \cdot \nu_1 + \ldots + \bm_r \cdot \nu_r \\
& = & \sum_{i=1}^r \bm_i 
\left[e_{k_i}^{(i)} \left( c_i - \frac{\beta_{k_i}^{(i)}}{m} \right) + 
\sum_{j = \min (k_i -1,0)}^{k_i-1} e_j^{(i)} \left(\frac{\beta_{j+1}^{(i)}}{m} -
\frac{\beta_j^{(i)}}{m}\right) \right]
\end{eqnarray*}
where each of the integer $k_i\in\{0,\ldots,s_i\}$ is uniquely determined when $c_1 \cdots c_r < +\infty$ by the condition
\begin{center}
$\beta_{k_i}^{(i)} \leq m_i \cdot c_i < \beta_{k_i +1}^{(i)}$.
\end{center}
\end{lemma}

The main result of this note is the following:
\begin{theorem}\label{thm:main} Let $f$ and $g$ be two analytic function-germs $(\CC,\bo)\rightarrow(\C,0)$.
Let $f=f_1^{\bm_1}\cdots f_r^{\bm_r}$ and $g=g_1^{\bn_1}\cdots g_s^{\bn_s}$ be respectively the irreducible decompositions
of the functions $f$ and $g$.
Let $X_i$ be the zero locus of $f_i$ and $Y_j$ be the zero locus of $g_j$.

\smallskip
The functions $f$ and $g$ are  subanalytically bi-Lipschitz contact equivalent if, and only if, there exists a
bijection $\sigma:\{f_1,\ldots,f_r\} \to \{g_1,\ldots,g_{r(=s)}\}$ between the irreducible factors of $f$ and $g$ such that

\smallskip
1) The multiplicities of each corresponding factors are equal, that is $\bm_i = \bn_{\sigma(i)}$,

\smallskip
2) The Puiseux pairs of $f_i$ and $g_{\sigma (i)}$ are the same, and

\smallskip
3) for any pair $i,j$, the intersection numbers $(X_i,X_j)_\bo$ and $(Y_{\sigma (i)},Y_{\sigma (j)})_\bo$ are equal.

\medskip
In particular, $f$ and $g$ are
subanalytically bi-Lipschitz contact equivalent if, and only if, they are right topologically equivalent.
\end{theorem}

\begin{proof}
%
First assume that, $r=s$,

- the intersection numbers $(X_i,X_j)_\bo$ and $(Y_i,Y_j)_\bo$ are equal for any $i\neq j$ and,

- the Puiseux pairs of the functions $f_i$ and $g_1$ are equal and,

- the multiplicities $\bm_i$ and $\bn_i$ are equal, for $i=1,\dots r$.

\smallskip
From Theorem \ref{theor-curves} we deduce there exists $H\colon(\CC,\bo)\rightarrow(\CC,\bo)$ 
a subanalytic bi-Lipschitz homeomorphism such that $H(X_i)=Y_i$ for any $i=1,\dots r$.
For each $i=1,\ldots,r$, Proposition \ref{main-proposition} implies there exist positive constants $0<A_i<B_i<+\infty$
such that in a neighbourhood of the origin we find
$$
A_i |f_i| \leq |g_i \circ H| \leq B_i |f_i|.
$$
Thus the functions $f$ and $g$ are bi-Lipschitz contact equivalent (via $h$).

\smallskip
The general situation, by hypothesis, is easily deduced from the special one above,
since it consists only in changing the indexation of one of the family of irreducible
factors (of the corresponding zero loci and the corresponding multiplicity).

\medskip
Conversely, we assume now that there exists $H\colon(\CC,\bo)\rightarrow(\CC,\bo)$  
a subanalytic bi-Lipschitz homeomorphism such that there exist positive constants $A<B$ such that in a
neighbourhood of the origin the following inequalities hold true:
\begin{equation}\label{eq-same-size}
A|f| \leq |g\circ H| \leq  B |f|.
\end{equation}
We immediately find $H(X)=Y$ and $r=s$. Up to re-indexation of the branches $Y_i$,
we also have $H(X_i)=Y_i$ for $i=1,\dots,r$.
Using Theorem \ref{theor-curves} again we deduce that the intersection numbers $(X_i,X_j)_\bo$ and $(Y_i,Y_j)_\bo$
are equal for any $i\neq j$ (let us denote each such number by $I_{i,j}$), the Puiseux pairs of the function-germs $f_i$ and
$g_i$ are equal.
It remains to prove that the multiplicities $\bm_i$ and $\bn_i$ are also equal, for $i=1,\dots,r$.
In order to prove that $\bm_1=\bn_1$, let $C$ be any real-analytic half-branch such that
the contact $c=c(C,X_1)$ is sufficiently large (and finite).  
More explicitly it means that the others contacts $c(C,X_i)$, for $i=2,\dots,r$, are equal to 
the intersection number $I_{i,1}:=(X_i,X_1)_\bo$. 
Since $H$ is a subanalytic bi-Lipschitz homeomorphism such that $H(X_i)=Y_i$
for any $i=1,\dots,r$, the pre-image $H^{-1}(C)$ is still a real analytic half-branch.
Since bi-Lipschitz homeomorphisms preserve the contact, we deduce that $c=c(H^{-1} (C),Y_1)$ and each contact
$c(H^{-1} (C),Y_i)$ is equal to the contact $(Y_i,Y_1)_\bo$, for $i=2,\dots,r$.
In other words we see
\begin{equation}\label{eq-order-g}
\nu_g (H^{-1} (C)) = c \cdot \bn_1 + I_{2,1}\cdot \bn_2+ \ldots +I_{r,1}\cdot \bn_r
\end{equation}
and
\begin{equation}\label{eq-order-f}
\nu_f (C) = c \cdot \bm_1 + I_{2,1}\cdot \bm_2+ \ldots +I_{r,1}\cdot \bm_r.
\end{equation}
Combining Equation (\ref{eq-same-size}) from the hypothesis, with Equations (\ref{eq-order-g}) and
(\ref{eq-order-f}) we conclude that
$$c\bn_1+I_{2,1}\bn_2+\ldots+I_{r,1}\bn_r=c\bm_1+I_{2,1}\bm_2+\cdots+I_{r,1}\bm_r.$$
Since the half-branch $C$ can be chosen asymptotically arbitrarily close to $X_1$, its contact $c$ goes $+\infty$, and thus
we find $\bm_1=\bn_1$. The same procedure can be applied for each remaining $i=2,\ldots,r$, substituting $i$ for $1$,
thus we conclude that that
\begin{center}
$\bm_i=\bn_i \;$ for $i=1,\dots,r$,
\end{center}
thus proving what we wanted.
\end{proof}
The first immediate consequence of our main result is the following:
\begin{cor}
Let $f$ and $g$ be two analytic function-germs $(\CC,\bo)\rightarrow(\C,0)$.
They are  bi-Lipschitz contact equivalent if, and only if, they are subanalytically bi-Lipschitz contact equivalent.
\end{cor}
The second consequence is:
\begin{cor}
The subanalytic bi-Lipschitz contact equivalence classification of complex analytic plane function-germs
has countably many equivalence classes.
\end{cor}

\end{document}